\documentclass[12pt]{amsart}
\usepackage{amsmath, amssymb, amsfonts, amsthm, mathtools,cite}
\usepackage{hyperref}
\usepackage{tabularx}
\usepackage{booktabs}
\usepackage{caption}
\usepackage{aurical}

\edef\restoreparindent{\parindent=\the\parindent\relax}
\usepackage{parskip}
\restoreparindent
\newtheorem{thm}{Theorem}[section]
\newtheorem{lem}[thm]{Lemma}
\newtheorem{cor}[thm]{Corollary}
\newtheorem{prop}[thm]{Proposition}

\newtheorem*{prob*}{Open problem}

\theoremstyle{definition}

\newtheorem{defi}[thm]{Definition}

\theoremstyle{remark}

\newtheorem{rem}[thm]{Remark}
\newtheorem*{rem*}{Remark}

\DeclareMathOperator{\Aut}{Aut}
\DeclareMathOperator{\ind}{ind}

\newcommand{\kringel}{\mathbin{\raise0.5pt\hbox{$\scriptstyle\circ$}}}
\newcommand{\pkt}{\mathbin{\raise0.5pt\hbox{$\scriptstyle\bullet$}}}
\newcommand{\sq}{\mathbin{\raise0.5pt\hbox{$\scriptscriptstyle\square$}}}

\newcommand{\R}{\mathbb{R}}
\newcommand{\Z}{\mathbb{Z}}

\newcommand{\ad}{{\rm ad}}
\newcommand{\Ad}{\mathop{\rm Ad}}

\newcommand{\Der}{{\rm Der}}

\newcommand{\Lf}{\mathfrak{f}}
\newcommand{\Lg}{\mathfrak{g}}

\newcommand{\Lz}{\mathfrak{z}}

\newcommand{\CF}{\mathcal{F}}

\newcommand{\abs}[1]{\lvert#1\rvert}

\newcommand{\al}{\alpha}
\newcommand{\be}{\beta}
\newcommand{\ga}{\gamma}
\newcommand{\de}{\delta}

\newcommand{\la}{\lambda}

\newcommand{\ov}{\overline}

\newcommand{\ra}{\rightarrow}

\renewcommand{\phi}{\varphi}

\newcommand{\Hpi}{\mathcal{H}_{\pi}}

\begin{document}


\title[Flat coadjoint orbits]{Characteristically nilpotent Lie groups \\ with flat coadjoint orbits}


\author[D. Burde]{Dietrich Burde}
\author[J. T. van Velthoven]{Jordy Timo van Velthoven}
\address{Fakult\"at f\"ur Mathematik\\
Universit\"at Wien\\
  Oskar-Morgenstern-Platz 1\\
  1090 Wien \\
  Austria}
\email{dietrich.burde@univie.ac.at}
\address{Fakult\"at f\"ur Mathematik\\
Universit\"at Wien\\
  Oskar-Morgenstern-Platz 1\\
  1090 Wien \\
  Austria}
\email{jordy.timo.van.velthoven@univie.ac.at}

\date{\today}

\subjclass[2000]{Primary 17A36, 22E27, Secondary 17B30}
\keywords{Characteristically nilpotent Lie algebras, dilations, flat coadjoint orbits}

\begin{abstract}
We study the existence of certain characteristically nilpotent Lie algebras with flat 
coadjoint orbits. Their connected, simply connected Lie groups admit square-integrable 
representations modulo the center. There are many examples of nilpotent Lie groups admitting 
families of dilations and square-integrable representations. Much less is known
about examples admitting square-integrable representations for which the quotient by the 
center does not admit a family of dilations. In this paper we construct a two-parameter 
family of characteristically nilpotent Lie groups $G(\al,\be)$ in dimension $11$, admitting 
square-integrable representations modulo the center $Z$, such that $G(\al,\be)/Z$ does 
not admit a family of dilations.  
\end{abstract}

\maketitle

\section{Introduction}
Two classes of nilpotent Lie algebras that have been studied in different areas of mathematics
are the so-called \emph{characteristically nilpotent Lie algebras} and the \emph{Lie algebras admitting an expanding automorphism}.
Characteristically nilpotent Lie algebras are Lie algebras for which all derivations are nilpotent. Such Lie algebras are
nilpotent. They have been studied in various contexts. Among other things, they 
provide examples of nilpotent Lie groups not admitting left-invariant affine structures, see, e.g., \cite{BU5, MIL}.
On the other hand, Lie algebras admitting an expanding automorphism, i.e., an automorphism all whose eigenvalues
$\lambda$ satisfy $|\lambda|>1$,  are of importance in harmonic analysis as they provide nilpotent Lie groups admitting
families of dilations, so-called \emph{homogeneous groups} \cite{FST}. The classes of characteristically nilpotent Lie algebras
and Lie algebras admitting expanding automorphisms are well-known to be disjoint from each other
(see Section \ref{sec:automorphims}). As a matter of fact, one of the first examples of a nilpotent Lie algebra
not admitting an expanding automorphism was a characteristically nilpotent
Lie algebra \cite{DYE}; see also \cite{dixmier1957derivations}. 

In the present paper we study the existence of certain characteristically nilpotent Lie algebras with flat coadjoint orbits.
Nilpotent Lie algebras with flat coadjoint orbits form a key object of study in the representation theory of nilpotent
Lie groups \cite{COG} as they characterize precisely those Lie algebras whose connected, simply connected nilpotent Lie groups
admit square-integrable representations modulo the center (see Section \ref{sec:squareintegrable}).
Whereas nilpotent Lie groups admitting both families of dilations and square-integrable representations exist in
abundance (see, e.g., \cite{corwin1983criteria}), there appears to be little known in the literature about the
existence of nilpotent Lie groups admitting square-integrable representations but not a family of dilations.
In particular, examples of characteristically nilpotent Lie groups $G$ with square-integrable
representations modulo the center $Z$ for which the quotient $G/Z$ is again characteristically nilpotent appear to be rare, see Remark \ref{rem}.
The existence of such examples is, however, of interest in view of recent papers
\cite{oussa2024orthonormal, grochenig2018orthonormal, grochenig2020balian} on orthonormal bases in the orbit of
square-integrable representations of nilpotent Lie groups as we explain next.

In \cite{grochenig2018orthonormal}, it was asked, given an irreducible, square-integrable representation
$(\pi, \Hpi)$ of a connected, simply connected nilpotent Lie group $G$, whether there always exists a
vector $\eta \in \Hpi$ and a discrete set $\Gamma \subseteq G$ such that $\{ \pi (\gamma) \eta \}_{\gamma \in \Gamma}$ forms an
orthonormal basis for $\Hpi$. This question was solved affirmatively in \cite{grochenig2018orthonormal} for graded
Lie groups with a one-dimensional center and was solved in \cite{oussa2024orthonormal} for arbitrary connected,
simply connected nilpotent Lie groups. It should be mentioned that the contruction of orthonormal bases
$\{ \pi (\gamma) \eta \}_{\gamma \in \Gamma}$ actually happens in $G/Z$, in the sense that if
$\{ \pi (\gamma) \eta \}_{\gamma \in \Gamma}$ is an orthonormal basis for $\Hpi$, then so is
$\{ \pi (\gamma \zeta_{\gamma} ) \eta \}_{\gamma \in \Gamma}$ for any sequence
$\{\zeta_{\gamma} \}_{\gamma \in \Gamma}$ in $Z$. Therefore, the investigation of orthonormal bases in the orbit of $G$
under $\pi$ is the same as investigating orthonormal bases in the orbit of associated projective representations
of $G/Z$ (see Section \ref{sec:squareintegrable}). For such projective representations, it was shown in
\cite{grochenig2020balian} that, if  the nilpotent group $G/Z$ admits a family of dilations, then
orthonormal bases $\{ \pi (\gamma) \eta \}_{\gamma \in \Gamma}$ for $\Hpi$ cannot consist of smooth vectors, that is, the
map $x \mapsto \pi(x) \eta$ cannot be smooth.
Although it is expected that smooth orthonormal bases do not exist for general connected, simply connected nilpotent
Lie groups (see \cite{grochenig2020balian, VV}), it appears that relevant examples of groups $G$ admitting
square-integrable representations but such that $G/Z$ does \emph{not} admit a family of dilations are rare in the literature. The aim of the present note is to construct a family of such examples. More precisely, we will
show the following theorem.

\begin{thm} \label{thm:intro}
There exists a two-parameter family of characteristically nilpotent Lie groups $G(\al,\be)$
in dimension $11$, admitting square-integrable representations modulo the center $Z$, such that
$G(\al,\be)/Z$ is characteristically nilpotent, and hence does not admit a family of dilations.
\end{thm}

A family of examples yielding Theorem \ref{thm:intro} is explicitly given in Section \ref{sec:example}. It consists
of so-called \emph{filiform} nilpotent Lie groups, meaning that the nilpotency index is maximal with respect to the dimension.
There are various obstructions to constructing such examples with the properties asserted in Theorem \ref{thm:intro}.
First, if $G$ admits a square-integrable representation modulo its center $Z$, then $G/Z$ is necessarily even-dimensional.
Since a filiform nilpotent Lie group has a one-dimensional center, this means that $G$ needs to have an odd dimension.
Second, if $\dim(G) \leq 7$, then $G/Z$ cannot be characteristically nilpotent, so that necessarily $\dim(G) \geq 9$
whenever $G$ admits square-integrable representations and $G/Z$ is characteristically nilpotent.
To find two-parameter families of Lie groups with such properties, it is more convenient to require that $\dim(G) \geq 11$ as
this allows for the construction of examples for which the required properties are easier to verify.

It should be noted that one could construct many more filiform nilpotent examples in any odd dimension $n\ge 11$ with the properties asserted in Theorem \ref{thm:intro}.
In fact, the filiform nilpotent Lie algebras $\Lf_n$ for $n\ge 13$, defined in section $5$ of \cite{BU35}, and their Lie groups
are good candidates for it. We expect that all Lie algebras $\Lf_n$ and their quotients $\Lf_n/\Lz$ are characteristically
nilpotent. 

For the variety $\CF_n$ of filiform nilpotent Lie algebras of dimension $n$ with $n\ge 7$ it is known, that
every irreducible component contains a nonempty Zariski open dense subset consisting of characteristically nilpotent
Lie algebras. For a reference see Theorem $11$ in \cite{ANC}. In other words, characteristically nilpotent filiform
Lie algebras are quite common. 

The organization of this paper is as follows. Section \ref{sec:automorphims} collects various notions and facts on automorphisms,
gradings and characteristically nilpotent Lie algebras. Section \ref{sec:squareintegrable} provides the relevant
background on square-integrable representations of nilpotent Lie groups. The flat coadjoint orbit condition is explained.
In Section \ref{sec:example} we construct the family of filiform nilpotent Lie groups
and Lie algebras for Theorem \ref{thm:intro}, and prove our main result on the level of Lie algebras and their
derivation algebras.

\section{Automorphisms and gradings} \label{sec:automorphims}

We will assume that all Lie groups and Lie algebras are finite-dimensional over the field 
of real numbers. We recall some notions and results that have appeared in different contexts for nilpotent
Lie groups and Lie algebras.

\begin{defi}
An {\em $\R$-grading} on a Lie algebra $\Lg$ is a decomposition of $\Lg$ as a direct sum
\[
\Lg=\bigoplus_{r\in \R}\Lg_r 
\]
of subspaces $\Lg_r\subseteq \Lg$ such that $[\Lg_r,\Lg_s]\subseteq \Lg_{r+s}$ for all $r,s\in \R$.
Here almost all of the spaces $\Lg_r$ are zero. We call such an $\R$-grading {\em positive}, if
$\Lg_r=0$ for all $r\le 0$.
\end{defi}  

In an analogous way one defines a $\Z$-grading on $\Lg$. 
Note that an automorphism $\phi\in \Aut(\Lg)$ preserves a grading if $\phi(\Lg_r)=\Lg_r$ for every $r$.
We call an automorphism $\phi\in \Aut(\Lg)$ {\em expanding}, if all eigenvalues $\la_i$ have absolute value
$\abs{\la_i}>1$. Every Lie algebra $\Lg$ admits the {\em trivial grading}, which is given by $\Lg=\Lg_0$.

\begin{defi}
A Lie algebra $\Lg$ is called {\em characteristically nilpotent}, if all derivations $D\in \Der(\Lg)$ are
nilpotent.  
\end{defi}  

In particular, a characteristically nilpotent Lie algebra $\Lg$ has only nilpotent inner derivations
$D=\ad(x)$ for all $x\in \Lg$, and hence is nilpotent by Engel's theorem. Moreover the derivation algebra $\Der(\Lg)$
of a characteristically nilpotent Lie algebra $\Lg$ is nilpotent.

The following lemma is stated in Dyer's paper \cite{DYE}, but without a proof.
Since we could not find a proof in the literature, we will give one here.

\begin{lem}\label{2.3}
Let $\Lg$ be a characteristically nilpotent Lie algebra.  Then the eigenvalues of any automorphism in $\Aut(\Lg)$ are
roots of unity.
\end{lem}  

\begin{proof}
Denote by $\Aut(\Lg)$ the automorphism group of $\Lg$, and by  $\Aut(\Lg)^0$ the connected component of the identity.
Since $\Aut(\Lg)$ is a linear algebraic group, it has only finitely many components in
the standard topology. Hence $\Aut(\Lg)^0$ has finite index in $\Aut(\Lg)$. The Lie algebra of both $\Aut(\Lg)$ and
$\Aut(\Lg)^0$ is given by the derivation algebra $\Der(\Lg)$. The exponential map determines a bijection between $\Der(\Lg)$
and $\Aut(\Lg)^0$ as $\Der(\Lg)$ is nilpotent. Since $\Lg$ is characteristically nilpotent, every derivation in
$\Der(\Lg)$ is nilpotent, so that every automorphism in $\Aut(\Lg)^0$ will only have eigenvalues equal to $1$.
Now let $\phi \in \Aut(\Lg)$. Since
$\Aut(\Lg)^0$ has finite index in $\Aut(\Lg)$, there exists an integer $k>0$ such that $\phi^k\in \Aut(\Lg)^0$. For any eigenvalue
$\la$ of $\phi$ it follows that $\la^k=1$.
\end{proof}  

For the following result see Lemma $2$ in \cite{BU9}.

\begin{lem}\label{2.4}
Let $\Lg$ be a characteristically nilpotent Lie algebra. Then the only $\Z$-grading on $\Lg$ is the trivial grading.
\end{lem}

\begin{proof}
Suppose that
\[
 \Lg=\bigoplus_{i\in \Z}\Lg_i 
\]
is a grading. Consider the bijective linear map $\phi\colon \Lg\ra \Lg$ defined by $\phi(x)=2^ix$ for all
$x\in \Lg_i$. The definition of a grading implies that $\phi$ is a Lie algebra automorphism of $\Lg$. Since the real eigenvalues
are equal to $1$ by Lemma $\ref{2.3}$, we have $i=0$ and  $\Lg=\Lg_0$.
\end{proof}  

Note that the converse statement also holds. If $\Lg$ admits a nontrivial $\Z$-grading, then the linear map $D:\Lg\ra \Lg$
defined by $D(x)=ix$ for all $x\in \Lg_i$ is a nonnilpotent derivation. 

Following \cite{FST}, we define dilations as follows.

\begin{defi}
A Lie algebra $\Lg$ is said to admit a {\em family of dilations}, if there exists a one-parameter family $(\de_t)_{t>0}$
of automorphisms $\de_t\in \Aut(\Lg)$ of the form
\[
\de_t=\exp(A\log(t)), \quad t>0,
\]
for a diagonalizable linear operator $A\colon \Lg\ra \Lg$ with positive eigenvalues.
\end{defi}  

The following result relates the various notions introduced above. It can be found in the literature among different references.
We summarize it here for the convenience of the reader.

\begin{prop}\label{2.6}
Let $\Lg$ be a Lie algebra. Then the following statements are equivalent.
\begin{itemize}
\item[(1)] $\Lg$ admits an expanding automorphism. 
\item[(2)] $\Lg$ admits a nontrivial positive $\R$-grading.
\item[(3)] $\Lg$ admits a nontrivial positive $\Z$-grading.
\item[(4)] $\Lg$ admits a family of dilations.
\end{itemize}  
If one of the conditions is satisfied then $\Lg$ is nilpotent.
\end{prop}  

\begin{proof}
By Theorem $2.7$ in \cite{DEL}, $(1)$ is equivalent to $(3)$. In addition, it is shown in Proposition $2.9$ of \cite{DEL},
that   $(1)$ implies that $\Lg$ is nilpotent. Hence it suffices to prove the equivalence of $(2),(3)$ and $(4)$. 
Assume that $(4)$ holds and we have a family of dilations $(\de_t)_{t>0}$ with $\de_t=\exp(A\log(t))$.
Then we obtain a decomposition $\Lg=\Lg_1\oplus \cdots \oplus \Lg_k$, where $\Lg_i$ denotes
the eigenspace of $A$ corresponding to the $i$-th positive eigenvalue $r_i$ of $A$. For every $r>0$ different from the eigenvalues
of $A$ set $\Lg_r=0$. Then $\de_t(x)=t^rx$ for any $x\in \Lg_r$. Let $x\in \Lg_r,y\in \Lg_s$ with $r,s>0$. We have
\[
\de_t([x,y])=[\de_t(x),\de_t(y)]=t^{r+s}[x,y],
\]
which implies $[\Lg_r,\Lg_s]\subseteq \Lg_{r+s}$. Thus $\Lg$ admits a positive $\R$-grading and $(2)$ holds.
By Proposition $2.6$ in \cite{DEI} condition $(2)$ implies condition $(3)$; see also \cite{miller1980parametrices} for another
proof that (4) implies (3). Finally suppose that $(3)$ holds, and $\Lg=\oplus_{i\in \Z} \Lg_i$
is a positive $\Z$-grading. Define the linear operator $A$ by $Ax=ix$ for $x\in \Lg_i$. This yields a family of dilations
$(\de_t)_{t>0}$ given by $\de_t=\exp(A\log(t))$ for $t>0$, and $(4)$ is satisfied.
\end{proof}  

\begin{cor}\label{2.7}
Let $\Lg$ be a characteristically nilpotent Lie algebra. Then $\Lg$ does not admit a family of dilations.
\end{cor}  

\begin{proof}
By Lemma \ref{2.4}, a characteristically nilpotent Lie algebra only admits the trivial $\Z$-grading. Hence, the claim follows from Proposition \ref{2.6}.
\end{proof}

Let $G$ be a connected, simply connected nilpotent Lie group with Lie algebra $\Lg$.
Then $\exp\colon \Lg\ra G$ is a diffeomorphism, and given a family of dilations $(\de_t)_{t>0}$ on $\Lg$ we define
a corresponding family of dilations on $G$ by $\exp\circ \de_t\circ \exp^{-1}$. Then we say that
{\em $G$ admits a family of dilations}, if $\Lg$ does admit one. Lie groups admitting a family of dilations are often called
{\em homogeneous groups}, see \cite{FST}.

\begin{rem}
The positivity of the grading in Proposition $\ref{2.6}$ is needed for the existence of a family of dilations.
It is worth pointing out that the following conditions are equivalent, see \cite{DEL}, Theorem $2.7$:
\begin{itemize}
\item[(i)] $\Lg$ admits an $\R$-grading with $\Lg_0=0$.
\item[(ii)] $\Lg$ admits a $\Z$-grading with $\Lg_0=0$.
\item[(iii)] $\Lg$ admits either an expanding or a hyperbolic automorphism, or both.
\end{itemize}
Here an automorphism of $\Lg$ without eigenvalues of modulus $1$ is called {\em hyperbolic}, if
it has at least one eigenvalue $\la$ with $\abs{\la}>1$ and at least one eigenvalues $\mu$ with $\abs{\mu}<1$.
Over the complex numbers the conditions are equivalent to the fact that $\Lg$ does not admit a {\em nonsingular derivation}.
Each condition implies that $\Lg$ is nilpotent. In this context it is interesting to note that there exist
nilpotent Lie algebras, which are neither characteristically nilpotent, nor admit a nonsingular derivation.
\end{rem}

\section{Nilpotent Lie groups admitting square-integrable representations} \label{sec:squareintegrable}

In this section we let $G$ always be a connected, simply connected nilpotent Lie group with Lie algebra $\Lg$.
Denote the dual vector space of $\Lg$ by $\Lg^{\ast}$ and the center of $\Lg$ by $\Lz$.  Let ${\rm Ad}\colon G\ra GL(\Lg)$
denote the the adjoint representation of $G$. Then $G$ acts on $\Lg^*$ by the {\em coadjoint map} 
${\rm Ad}^*\colon G\ra GL(\Lg^*)$ defined by
\[
({\rm Ad}^*(g) \ell )(x)=\ell({\rm Ad}(g)^{-1}x),
\]
for $x\in \Lg$, $g\in G$ and $\ell\in \Lg^*$.

\begin{defi}
The stabilizer of $\ell\in \Lg^*$ for the coadjoint map is the closed connected subgroup
\[
G(\ell)=\{g\in G\mid {\rm Ad}^*(g)\ell=\ell \}
\]
of $G$. Its Lie algebra is given by the Lie subalgebra
\[
\Lg(\ell)=\{y\in \Lg\mid \ell([x,y])=0 \; \forall x\in \Lg\}
\]
of $\Lg$. 
\end{defi}

Each $\ell\in \Lg^*$ determines a skew-symmetric bilinear form $B_{\ell}\colon \Lg\times \Lg\ra \R$ by $B_{\ell}(x,y)=\ell([x,y])$.
The radical of $B_{\ell}$ is given by $\Lg(\ell)$. Therefore $\Lg(\ell)$ is sometimes called the {\em radical of $\ell$}.
Note that $\Lg(\ell)$ contains the center $\Lz$ of $\Lg$. For any $\ell\in \Lg^*$, the subalgebra $\Lg(\ell)$ has even codimension
in $\Lg$, see Lemma $1.3.2$ in \cite{COG}.

\begin{rem}
The radical $\Lg(\ell)$ of $\ell\in \Lg^*$ is discussed in \cite{COG} with several examples. It is also
related to the {\em index} of a Lie algebra $\Lg$, which is defined by
\[
\chi(\Lg)=\inf \{ \dim \ell(\Lg) \mid \ell \in \Lg^*\}.
\]
For a discussion and computation of the index see \cite{AMA} and the references given there.
\end{rem}  

Let $(\pi, \Hpi)$ be an irreducible unitary representation of $G$. Then there exists $\ell \in \mathfrak{g}^*$ and a
polarization $\mathfrak{h}$ of $\ell$ such that $\pi$ is unitarily equivalent to an induced representation
$\pi_{\ell} := \ind_H^G \chi_{\ell}$, where $H = \exp(\mathfrak{h})$ and
\[
 \chi_{\ell} (\exp X) = e^{i \ell (X)}, \quad X \in \mathfrak{h}.
\]
Two irreducible induced representations $\ind_H^G (\chi_{\ell})$ and $\ind_H^G (\chi_{\ell'})$ are
unitarily equivalent if and only if $\ell, \ell' \in \mathfrak{g}^*$ belong to the same coadjoint orbit.
The coadjoint orbit $\mathcal{O}_{\ell} := \Ad^*(G) \ell$ associated with the equivalence class $\pi \in \widehat{G}$ is
often simply denoted by $\mathcal{O}_{\pi}$. For more details see Section $2.2$ in \cite{COG}.

\begin{defi}
An irreducible unitary representation $(\pi, \Hpi)$ of $G$ is said to be \emph{square-integrable modulo $Z$} if it
admits a vector $\eta \in \Hpi \setminus \{0\}$ such that
\begin{align*} 
 \int_{G/Z} |\langle \eta, \pi(\ov{x}) \eta \rangle |^2 \; d\mu_{G/Z} (\ov{x}) < \infty,  
\end{align*}
where $\ov{x} = xZ$ and $\mu_{G/Z}$ denotes the Haar measure on $G/Z$.
\end{defi}

If $(\pi, \Hpi)$ is square-integrable modulo $Z$,
then it  can be treated as projective unitary representation $\overline{\pi}$ of $G/Z$. Indeed, if $s : G/Z \to G$
is a continuous section for the canonical projection $q : G \to G/Z$, meaning that
$q \circ s = \mathrm{id}_{G/Z}$, then the map $\overline{\pi} : G/Z \to \mathcal{U} (\Hpi)$ defined by
\begin{align*}  
 \overline{\pi} (\ov{x}) = \pi(s(\ov{x})) \quad \text{for} \quad \ov{x} = xZ \in G/Z, 
\end{align*}
is a projective unitary representation of $G/Z$ that is irreducible and square-integrable. 

The following characterization is a combination of Theorem $3.2.3$ and Corollary $4.5.4$ in \cite{COG}.

\begin{lem} \label{lem:flat_orbit}
Let $(\pi, \Hpi)$ be an irreducible unitary representation of $G$.
The following assertions are equivalent:
\begin{enumerate}
 \item[(1)] $\pi$ is square-integrable modulo $Z$;
 \item[(2)] $\mathfrak{g} (\ell) = \mathfrak{z}$ for some (all) $\ell \in \mathcal{O}_{\pi}$;
 \item[(3)] $\mathcal{O}_{\ell} = \ell + \mathfrak{z}^{\perp}$ for some (all) $\ell \in \mathcal{O}_{\pi}$.
\end{enumerate}
\end{lem}

The condition $(3)$ is often referred to as the \emph{flat orbit} condition.

\section{A two-parameter family} \label{sec:example}

In this section we will construct a two-parameter family of nilpotent Lie groups yielding Theorem \ref{thm:intro}. In addition, this yields nilpotent Lie groups admitting
square-integrable projective unitary representations, but does not admit a family of dilations. This follows from the
following lemma, which transfers the question to the level of Lie algebras.

\begin{lem} \label{4.1}
Let $\mathfrak{g}$ be  a nilpotent Lie algebra admitting a linear form $\ell \in \mathfrak{g}^*$ such that
$\mathfrak{g}(\ell) = \mathfrak{z}$ and such that $\mathfrak{g} / \mathfrak{z}$ is characteristically nilpotent.
Then the associated connected, simply connected nilpotent Lie group with Lie algebra $\mathfrak{g}/\mathfrak{z}$ admits
irreducible, square-integrable projective unitary representations but does not admit a family of dilations.
\end{lem}
\begin{proof}
Let $G$ be the connected, simply connected Lie group with Lie algebra $\mathfrak{g}$. If there exists $\ell \in \mathfrak{g}^*$ such
that $\mathfrak{g}(\ell) = \mathfrak{z}$, then the irreducible representation $\pi = \pi_{\ell}$ of $G$ associated to $\ell$ is
square-integrable modulo the center $Z$ of $G$ by Lemma \ref{lem:flat_orbit}. As such, it gives rise to an irreducible,
square-integrable projective representation $\overline{\pi}$ of $G/Z$ as indicated above.
However, since the Lie algebra $\mathfrak{g}/\mathfrak{z}$ of $G/Z$ is characteristically nilpotent, it does not admit a
family of dilations by Corollary $\ref{2.7}$. Thus, $G/Z$ does not admit a family of dilations.
\end{proof}

In order to find examples yielding Theorem \ref{thm:intro}, we need to construct nilpotent Lie algebras $\Lg$ together with
 $\ell\in \Lg^*$ such that $\Lg(\ell)=\Lz$, so that the quotient algebra $\Lg/\Lg(\ell)$ is characteristically nilpotent.
Note that $\Lg/\Lg(\ell)$ has to be even-dimensional. A natural choice here for $\Lg$ are
so called {\em filiform nilpotent} Lie algebras, i.e., nilpotent Lie algebras of dimension $n$ with nilpotency index $n-1$.
For filiform nilpotent Lie algebras the center is one-dimensional, so that $\dim (\Lg)$ needs to be odd.
Note that we need $\dim \Lg\ge 9$, because for $\dim \Lg\le 7$ the quotient Lie algebra $\Lg/\Lz$ has dimension less than $7$, and hence
admits a nonsingular derivation and cannot be characteristically nilpotent. In order to find multi-parameter families
we even would like to assume that $\dim \Lg\ge 11$. We have already studied filiform nilpotent Lie algebras of dimension $10$ and $11$
in \cite{BU5} in the context of affinely-flat structures. This was helpful for the  construction of the following
two-parameter family $\Lg=\Lg(\al,\be)$ of filiform nilpotent Lie algebras of dimension $11$, with an adapted basis
$\{e_1,\ldots ,e_{11}\}$. It is defined by the following Lie brackets:

\begin{align*}
[e_1,e_i]   & = e_{i+1} \text{ for all } 1\le i\le 10,\\[0.1cm]
[e_2,e_3]   & = e_5+\al e_6, \\ 
[e_2,e_4]   & = e_6+\al e_7, \\  
[e_2,e_5]   & = -e_7+(\al-\be)e_8,\\  
[e_2,e_6]   & = -3e_8+(\al-2\be)e_9,\\  
[e_2,e_7]   & = -2e_9-\frac{1}{4}(5\al+7\be)e_{10}+\frac{1}{16}(27\al^2+12\al\be+\be^2)e_{11}, \\ 
[e_2,e_8]   & = 2e_{10}-\frac{1}{4}(23\al+\be)e_{11},\\
[e_2,e_9]   & = -e_{11},\\[0.1cm]
[e_3,e_4]   & = 2e_7+\be e_8, \\  
[e_3,e_5]   & = 2e_8+\be e_9,\\  
[e_3,e_6]   & = -e_9+\frac{1}{4}(9\al-\be)e_{10}-\frac{1}{16}(27\al^2+12\al\be+\be^2)e_{11},\\  
[e_3,e_7]   & = -4e_{10}+\frac{3}{2}(3\al-\be)e_{11}, \\ 
[e_3,e_8]   & = 3e_{11},\\[0.1cm]
[e_4,e_5]   & = 3e_9-\frac{1}{4}(9\al-5\be)e_{10}+\frac{1}{16}(27\al^2+12\al\be+\be^2)e_{11}\\  
[e_4,e_6]   & = 3e_{10}-\frac{1}{4}(9\al-5\be)e_{11},\\  
[e_4,e_7]   & = -7e_{11}, \\
[e_5,e_6]   &= 10e_{11}.  
\end{align*}

The Jacobi identity holds for all parameters $\al,\be$. The Lie brackets, which are not 
implied by skew-symmetry, or bilinearity, are zero.
Denote by $\Lz$ the center of these Lie algebras. We have $\Lz={\rm span} \{e_{11}\}$ for 
all pairs $(\al,\be)$.

\begin{prop} \label{4.2}
The Lie algebras $\Lg(\al,\be)$ are pairwise nonisomorphic, i.e., we have  
$\Lg(\al,\be)\cong \Lg(\al',\be')$ if and only if $(\al,\be)=(\al',\be')$. 
They are characteristically nilpotent if and only if $(\al,\be)\neq (0,0)$.
\end{prop}  

\begin{proof}
Let $\phi\colon \Lg(\al,\be)\ra  \Lg(\al',\be')$ be a Lie algebra homomorphism. Then the conditions easily imply
that the matrix of $\phi$ is lower-triangular with diagonal $(\xi,\xi^2,\ldots ,\xi^{11})$ for some $\xi$, with respect to the
adapted basis $\{e_1,\ldots ,e_{11}\}$. The equations in the entries of the matrix for $\phi$ 
immediately imply that $\al'=\al$ and $\be'=\be$. A computation shows that the derivation algebra
of $\Lg(\al,\be)$ is $13$-dimensional for all $(\al,\be)$ with $3\al+\be\neq 0$, or with $3\al+\be=0$, $\al\neq 0$.
All derivations are represented by a strictly lower-triangular matrix in this case, and hence are all nilpotent.
A basis of derivations can be written down as in the proof of Proposition $\ref{4.4}$.
\end{proof}  

\begin{prop} \label{4.3}
Let $\Lg=\Lg(\al,\be)$ and $\ell=e_{11}^* \in \mathfrak{g}^*$. Then we have $\Lg(\ell)=\Lz$ for all
pairs $(\al,\be)$.
\end{prop}  

\begin{proof}
Let $x=(x_1,\ldots ,x_{11})\in \Lg(\ell)$, which means $\ell ([x,y])=0$ for all $y\in \Lg$. This is equivalent to
the conditions $\ell ([x,e_i])=0$ for $i=1,\ldots 11$. For each $1\le i\le 10$ we obtain a linear equation in the
variables $x_1,\ldots ,x_{11}$. For $i=11$ we only have $0=0$. Thus we obtain the following system of linear equations.
\begin{align*}
0 & = x_{10},\\
0 & = \ga x_7 - 4(23\al+\be)x_8 - 16x_9, \\
0 & = \ga x_6 - 24(3\al-\be)x_7 - 48x_8,\\
0 & = \ga x_5 - 4(9\al - 5\be)x_6 - 112x_7,\\
0 & = \ga x_4 - 160x_6,\\
0 & = \ga x_3 + 4(9\al - 5\be)x_4 - 160x_5,\\
0 & = \ga x_2 + 24(3\al-\be)x_3 - 112x_4,\\
0 & = (23\al + \be)x_2 - 12x_3,\\
0 & = x_2,\\
0 & = x_1,  
\end{align*}
where $\ga=27\al^2 + 12\al\be + \be^2$. It follows that $x_1=\cdots =x_{10}=0$ and $x\in \Lz$.
So we have $\Lg(\ell)\subseteq \Lz$. The converse inclusion is clear. 
\end{proof}

\begin{prop}\label{4.4}
The quotient Lie algebras $\Lg(\al,\be)/\Lz$ of dimension $10$ are characteristically nilpotent if and only if
$(\al,\be)\neq (0,0)$.
\end{prop}
  
\begin{proof}
A direct computation shows that for $(\al,\be)\neq (0,0)$ all derivations are strictly 
lower-triangular with respect to the basis $\{e_1,\ldots ,e_{10}\}$. Hence all derivations are nilpotent and 
the quotient algebra is characteristically nilpotent. For the explicit computation of all derivations we make
a case distinction. Assume that  $3\al+\be \neq 0$. Then the derivation algebra has dimension $12$, and a basis of derivations
$D_1,\ldots D_{12}$ is given as follows. Here $E_{ij}$ denotes the matrix of size $10$ with entry $1$ at position $(i,j)$
and all other entries equal to zero.
\begin{align*}
D_1    & = E_{3,1}-E_{5,3}-\al E_{6,3}-E_{6,4}-\al E_{7,4}+E_{7,5}+(\be-\al) E_{8,5}+3E_{8,6}+(2\be-\al)E_{9,6} \\
         & \hspace*{0.5cm} + 2E_{9,7}+ \tfrac{5\al+7\be}{4}E_{10,7}-2E_{10,8} \\   
D_2    & = E_{4,1}+E_{5,2}+\al E_{6,2}-2E_{7,4}-\be E_{8,4}-2E_{8,5}-\be E_{9,5}+E_{9,6}+\tfrac{\be-9\al}{4}E_{10,6}+4E_{10,7} \\ 
D_3    & = E_{5,1}+E_{6,2}+\al E_{7,2}+2E_{7,3}+\be E_{8,3}-3E_{9,5}+\tfrac{9\al-5\be}{4}E_{10,5}-3E_{10,6} \\
D_4    & = E_{6,1}-E_{7,2}+2E_{8,3}+(2\be-\al)E_{9,3}+3E_{9,4}+\tfrac{9\be-13\al}{4}E_{10,4} \\ 
D_5    & = E_{7,1}+2E_{9,3}+\tfrac{5\al+7\be}{4}E_{10,3} \\ 
D_6    & = E_{8,1}-2E_{10,3} \\
D_7    & = E_{9,1} \\ 
D_8    & = E_{10,1} \\ 
D_9    & = E_{3,2}+E_{4,3}+E_{5,4}+E_{6,5}+E_{7,6}+E_{8,7}+E_{9,8}+ E_{10,9} \\ 
D_{10} & = E_{8,2}+E_{9,3}+ E_{10,4} \\ 
D_{11} & = E_{9,2}+E_{10,3} \\ 
D_{12} & = E_{10,2}
\end{align*}
For $3\al+\be=0$ with $\al\neq 0$ again all derivations are nilpotent and a basis is given as above.
For $(\al,\be)=(0,0)$ however, the derivation algebra has dimension $13$, and the additional basis vector is given by the
invertible derivation $D={\rm diag}(1,2,3,\ldots,10)$.
Hence the quotient algebra is not characteristically nilpotent in this case.
\end{proof}

Denote by $G(\al,\be)$ the connected, simply connected nilpotent Lie groups of dimension $11$ with Lie algebra
$\Lg(\al,\be)$. We denote by $Z$ the center of $G(\al,\be)$, and by $\Lz$ the center of $\Lg(\al,\be)$. With this notation, a combination of  Lemma \ref{lem:flat_orbit}, Proposition \ref{4.2}, Proposition \ref{4.3} and Proposition \ref{4.4} directly yields Theorem \ref{thm:intro}. In addition, the following is a direct consequence of Lemma \ref{4.1}.

\begin{thm}
The two-parameter family of nilpotent Lie groups $G(\al,\be)/Z$ admits irreducible, square-integrable projective unitary
representations, but does not admit a family of dilations for all pairs $(\al,\be)\neq (0,0)$.
\end{thm}  

\begin{rem} \label{rem}
After submitting this paper, we were informed by Ingrid Belti\c{t}\u{a} about
an arXiv preprint \cite{BBD1} in which an example is given of an $8$-dimensional
filiform symplectic Lie algebra that is characteristically nilpotent, cf. \cite[Example 7.5]{BBD1}. This yields a
$9$-dimensional example of a Lie algebra with properties as in our main theorem, see, e.g., \cite[Remark 7.2]{BBD1}. The example \cite[Example 7.5]{BBD1} is, however, not included in the
journal publications \cite{BBD3, BBD2} of the arXiv preprint \cite{BBD1}.
\end{rem}

\section*{Acknowledgments}
This research was funded by the Austrian Science Fund (FWF): Grant-DOI 10.55776/P33811, Grant-DOI 10.55776/PAT2545623 and Grant-DOI 10.55776/J4555. For open
access purposes, the authors have applied a CC BY public copyright license to any author-accepted manuscript version arising
from this submission.

\end{document}